\documentclass{article}
\usepackage{amsmath,amssymb,amsthm, tikz-cd}
\usepackage{graphicx}
\usepackage{enumitem}
\usepackage[margin=1in]{geometry}
\usepackage{dutchcal}
\usepackage{colonequals}
\usepackage[
backend=biber,
style=alphabetic,
sorting=ynt
]{biblatex}
\usepackage[colorlinks,  citecolor=black, linkcolor = black, urlcolor = black]{hyperref} 
\usepackage{forest}
\usepackage{thmtools}

\usepackage{stmaryrd}
\usepackage{tikz-cd}
\usepackage[nameinlink]{cleveref}

\usepackage{algorithm}
\usepackage{algpseudocode}
\algdef{SE}[SUBALG]{Indent}{EndIndent}{}{\algorithmicend\ }%
\algtext*{Indent}
\algtext*{EndIndent}

\theoremstyle{plain}
\newtheorem{theorem}{Theorem}[section]
\newtheorem{corollary}[theorem]{Corollary}

\newtheorem{proposition}[theorem]{Proposition}
\theoremstyle{definition}
\newtheorem{remark}[theorem]{Remark}

\newtheorem{example}[theorem]{Example}

\newtheorem{definition}[theorem]{Definition}

\setlist[enumerate]{label = (\roman*), topsep=0pt, itemsep=0pt}

\newlist{enuroman}{enumerate}{1}
\setlist[enuroman]{
	label = (\roman*),
	topsep = 0pt,
	labelindent = \parindent}
	
\newlist{enualph}{enumerate}{2}
\setlist[enualph]{
	label = (\alph*),
	topsep = 0pt,
	labelindent = \parindent}

\newlist{enuarabic}{enumerate}{3}
\setlist[enuarabic]{
	label = (\arabic*),
	topsep = 0pt,
	labelindent = \parindent}

\setlist[itemize]{label = $\diamond$, topsep=0pt, itemsep=0pt, labelindent=0pt}

\newcommand{\Z}{\mathbb{Z}}

\newcommand{\arrow}{arrow}
\newcommand{\F}{\mathbb{F}}

\newcommand{\N}{\mathbb{N}}

\newcommand{\fm}{\mathfrak{m}}

\newcommand{\ceq}{\colonequals}
\newcommand{\coloneqq}{\colonequals}

\DeclareMathOperator{\height}{ht}

\renewcommand{\phi}{\varphi}
\renewcommand{\to}{\longrightarrow}

\newcommand\blfootnote[1]{%
  \begingroup
  \renewcommand\thefootnote{}\footnote{#1}%
  \addtocounter{footnote}{-1}%
  \endgroup
}

\newcommand{\directory}[1]{
\begin{forest}
  for tree={
    font=\ttfamily,
    grow'=0,
    child anchor=west,
    parent anchor=south,
    anchor=west,
    calign=first,
    edge path={
      \noexpand\path [draw, \forestoption{edge}]
      (!u.south west) +(7.5pt,0) |- node[fill,inner sep=1.25pt] {} (.child anchor)\forestoption{edge label};
    },
    before typesetting nodes={
      if n=1
        {insert before={[,phantom]}}
        {}
    },
    fit=band,
    before computing xy={l=15pt},
  }
#1
\end{forest}
}

\let\cref\Cref
\let\autoref\cref

\usepackage{color}
\definecolor{chianti}{rgb}{0.6,0,0}
\definecolor{meretale}{rgb}{0,0,.6}
\definecolor{leaf}{rgb}{0,.35,0}

\begin{document}

\title{The WittVectors package for Macaulay2}
\author{Anne Fayolle, Abhay Goel, Devlin Mallory, Eamon Quinlan-Gallego, Teppei Takamatsu}
\maketitle

\blfootnote{
We would like to thank Karl Schwede for putting our team together and showing us how to work with Macaulay2.
We would also like to thank Kevin Tucker, Jakub Witaszek, and Nawaj KC for useful conversations on this topic. This project started during the academic year 2023-2024 as part of an RTG seminar under NSF RTG Grant No.~1840190. 
Fayolle was supported by an NSERC doctoral grant, a Simons dissertation fellowship and Simons Foundation SFI-MPS-MOV-00006719-07, 
Goel was supported by the RTG Grant No.~1840190,
Mallory was supported by the National Science Foundation under the RTG Grant No.~1840190, and by EUR2023-143443 funded by MCIN/AEI/10.13039/501100011033, 
as well as funding from the European Union's Horizon Europe research and innovation programme under the Marie Skłodowska-Curie grant agreement No 101202625.
Quinlan-Gallego was supported by an NSF postdoctoral fellowship \#2203065, the RTG grant \#2037569, and the Ram\'on y Cajal Fellowship RYC2024-049057-I.
Takamatsu was supported by JSPS KAKENHI Grant number JP25K17228.
}

\begin{abstract}

We implement a Macaulay2 package for computations involving rings of truncated Witt vectors $W_n(R)$ of a finitely generated $\F_p$-algebra $R$.  This package includes ring operations (addition, multiplication, Frobenius, Verschiebung, etc.) on elements of $W_n(R)$,
    conversion between tuple representatives and ghost map representatives of elements of truncated Witt vectors over polynomial rings, and
   explicit computation of $W_n(R)$ as a finite-type algebra over $\Z/p^n$. 
This functionality is based on an algorithm we develop for performing arithmetic operations in rings of finite length $p$-typical Witt vectors over finitely generated $\F_p$-algebras. 
In addition, we implement an algorithm to calculate lifts of Frobenius from $\F_p$-algebras to flat lifts over $\Z/p^2\Z$, and an algorithm to find the quasi-$F$-splitting height of a local or graded complete intersection ring.
\end{abstract}

\section{Introduction}

Fix a prime number $p > 0$.
Given a commutative ring $R$ and an integer $n \geq 1$, we let $W_{n}(R)$ be the ring of $p$-typical Witt vectors of length $n$ over $R$. 
Recall that, as a set, we have $W_n(R) = R \times R \times \cdots \times R$ ($n$ copies), with addition and multiplication defined by certain universal polynomials (see \cite[Section 8.10]{JacobsonBasicAlgebraII} for more details about these universal polynomials). Using these universal polynomials, one can write down explicitly how addition and multiplication work in $W_n(R)$. 
For the following explicit formulas, we assume that $R$ is an $\F_p$-algebra. For \(n=2\), we have

\[
(x_0,y_0)+(x_1,y_1)=\left(x_0+x_1, y_0+y_1-\frac{(x_0+x_1)^p-x_0^p-x_1^p}{p}\right)
\]
and 
\[
(x_0,y_0)(x_1,y_1)=\left(x_0x_1, x_0^py_1+x_1^py_0+py_0y_1\right).
\]
For $n=3$, we have 
\begin{align*}
(x_0,y_0,z_0)+(x_1,y_1,z_1)&=\Bigg(x_0+x_1, y_0+y_1-\frac{(x_0+x_1)^p-x_0^p-x_1^p}{p},\\
&z_0+z_1+\frac{1}{p}\left(y_0^p+y_1^p-\left(y_0+y_1+\frac{x_0^p+x_1^p-(x_0+x_1)^p}{p}\right)^p\right)+\frac{x_0^{p^2}+x_1^{p^2}-(x_0+x_1)^{p^2}}{p^2}\Bigg)
\end{align*}
and 
\[
(x_0,y_0,z_0)(x_1,y_1,z_1)=\left(x_0x_1, x_0^py_1+x_1^py_0+py_0y_1,z_0x_1^{p^2}+y_0^py_1^p+x_0^{p^2}z_1+\frac{y_0^px_1^{p^2}+x_0^{p^2}y_1^p-(y_0x_1^p+x_0^py_1)^p}{p}\right).
\]
As one can see from these examples, the complexity of these operations increases dramatically as $n$ gets bigger, and this makes computations quite difficult in $W_n(R)$ for  $n>2$. 
Indeed, historically, even computations with Witt vectors of length $3$ have been challenging, see for instance \cite{FinottiComputationsWittVectorsLength3}. To overcome these difficulties, we use Illusie's  \cite{IllusieComplexeDeRhamWittCristalline} presentation of the ring of Witt vectors of $\F_p[x_1,\dots,x_n]$ as a subalgebra of a polynomial ring over $\Z/p^n$ (see \autoref{thm-witt-presentation-poly}) to reduce to computations in a polynomial ring. 

This package implements the following basic capabilities for dealing with (truncated) rings of Witt vectors of a finitely generated $\F_p$-algebra $R$:
\begin{enumerate}
    \item Ring operations (addition, multiplication, Frobenius, Verschiebung, etc.) on elements of $W_n(R)$, represented as tuples of length $n$. 
    \item Conversion between tuple representatives and ghost map representatives of elements of $W_n(S)$ for $S$ a polynomial ring.
    \item Explicit computation of $W_n(R)$ as a finite-type algebra over $\Z/p^n$. 
    \item Computation of the  maps $W_n(R)\to W_n(R')$ induced by a ring map $R\to R'$.
\end{enumerate}

Using these Witt vector computations, we developed an algorithm that gives constraints to the possible lifts of Frobenius on a specific lift $B$ of a given ring $A$ of characteristic $p$ to characteristic $p^2$, see \autoref{section frobeniuslift}. We also implemented an algorithm that allows us to compute the quasi-$F$-splitting height of a complete intersection using the Fedder-type criterion of \cite{KawakamiTakamatsuYoshikawaFedderQuasi}, see \autoref{section quasiFSplitting}.

\begin{remark}
   We began our project in Fall 2023 and since then, a similar algorithm has been developed independently by Muñoz-\relax-Bertrand in Sage \cite{MunozBertrandFasterComputationWittVectors}. Beyond the additional capabilities of this package (e.g., explicit calculations of rings of Witt vectors, applications to quasi-$F$-splitting and Frobenius lifts) we believe it is useful to have addition and multiplication implemented in Macaulay2 due to the increasing use of Witt vectors by commutative algebraists. 
\end{remark}

\section{Finite-length Witt vectors as finitely generated algebras}

We start with some background on rings of Witt vectors. The material in this section is well-known, so we omit many proofs; the reader is invited to consult  \cite[Chapter 2, section 6]{SerreLocalFields} and \cite[Section 8.10]{JacobsonBasicAlgebraII} for more thorough expositions.
All rings considered here are unitary.

Witt \cite{WittZyklischeCharp} showed that for a commutative ring $R$ there is a naturally associated ring of Witt vectors $W(R)$.
Although there are many ways to think about this ring, including as a universal $\delta$-ring (see \cite{JoyalDeltaAnneauWitt}), we will focus on the $p$-typical Witt vectors with the Witt coordinates (as opposed to Joyal or ghost coordinates). 
That is, as a set,
\[
W(R)=R\times R \times R \times \ldots.
\]
 As for the ring structure, it is described by certain universal polynomials on the coordinates of elements of $W(R)$: $S_0, S_1, S_2,\ldots$ and $P_0, P_1, P_2,\ldots$. More precisely for $r=(r_0,r_1,r_2,\ldots)$ and $s=(s_0,s_1,s_2,\ldots)$ in $W(R)$, 
\[
r+s=(S_0(r,s),S_1(r,s),S_2(r,s),\ldots)
\]
and 
\[
r\cdot s=(P_0(r,s),P_1(r,s),P_2(r,s),\ldots).
\] 
Instead of working with infinite tuples, one can instead work with the \emph{truncated ring of $p$-typical Witt vectors} $W_n(R)$ for $n$ a positive integer, also known as the \emph{ring of Witt vectors of length $n$}. This is the truncation of $W(R)$ to its first $n$ coordinates so an element of $W_n(R)$ looks like $(r_0,r_1,r_2,\ldots,r_n)$ for $r_i\in R$. 
Mathematically, $W_n(R)$ is a quotient ring of $W(R)$ by the ideal of tuples with zeroes in the first $n$ coordinates;
addition and multiplication are thus truncation of the same operations in $W(R)$. More precisely if $(r_0,r_1,r_2,\ldots,r_{n-1})$ and $(s_0,s_1,s_2,\ldots,s_{n-1})\in W_n(R)$, then we can choose lifts $r\coloneqq (r_0,r_1,r_2,\ldots,r_{n-1}, 0, 0, \ldots)$ and $s\coloneqq (s_0,s_1,s_2,\ldots,s_{n-1}, 0, 0, \ldots)$  in $W(R)$ and so  define 
\[
(r_0,r_1,r_2,\ldots,r_{n-1})+(s_0,s_1,s_2,\ldots,s_{n-1})\coloneqq (S_0(r,s),S_1(r,s),S_2(r,s),\ldots S_{n-1}(r,s))
\]
and 
\[
(r_0,r_1,r_2,\ldots,r_{n-1})\cdot(s_0,s_1,s_2,\ldots,s_{n-1})\coloneqq (P_0(r,s),P_1(r,s),P_2(r,s),\ldots,P_{n-1}(r,s)).
\]
Note that $W(R)$ is the projective limit of $W_n(R)$ as $n\to \infty$ and $W_1(R)=R$. 

The rings of Witt vectors come with two main operations: the Verschiebung, (shift in German) and the Frobenius (which has a particularly clean description when $R$ has characteristic $p$).

\begin{definition}[Verschiebung]
Let $m\in \N_{>0}$ be arbitrary. 
We define 
\[
V^m\colon W_n(R)\to W_{n+m}(R)\quad (r_0,r_1,r_2,\ldots,r_{n-1})\mapsto (0,\ldots, 0, r_0,r_1,r_2,\ldots,r_{n-1}),
\]
with $m$ zeros. We write $V$ when $m=1$. Note that this is an additive map.
\end{definition}

\begin{definition}[Frobenius]
Let $R$ be commutative of characteristic $p>0$. For $n=0,1,2,\ldots,\infty$, the Frobenius on $W_n(R)$ is $F\colon W_n(R)\to W_n(R)$, $(r_0,r_1,r_2,\ldots)\mapsto (r_0^p,r_1^p,r_2^p,\ldots)$.
\end{definition}

Frobenius and Verschiebung commute with each other, and their composition is equal to multiplication by $p$ on $W_n(R)$.

Since truncated Witt vectors $W_n(R)$ form a ring, one can discuss ideals. Two types of ideals that arise naturally are 
\[
W_n(I):=\{(r_0,\ldots,r_{n-1})\in W_n(R):r_i\in I\text{ for all }i\}.
\]
for an ideal $I$ of $R$, and kernels of the truncation maps $W_{n+m}(R) \to W_n(R)$, or equivalently the images of Verschiebung maps 
\[
\ker\bigl(W_{n+m}(R)\to W_{n}(R)\bigr)=V^nW_{m}(R)
\]
for $n,m=0,1,2,\ldots,\infty$.
These are of course not the only types of ideals; one can specify a set of generators to give an ideal, which we do in our implementation (see section \ref{section witt ideals} for more details).

\begin{definition}
For any commutative ring $R$ and any $n\geq 0$, $W_{n+1}(R)$ is equipped with a natural ring map to $R$, the $n$-th \emph{ghost polynomial} 
\begin{equation}\label{ghost map def}
w_n(x_0,\ldots,x_n)
=
x_0^{p^n}+px_1^{p^{n-1}}+\cdots+p^{n-1}x_{n-1}^p+p^nx_n.
\end{equation}
This gives rise to the \emph{ghost map} $\operatorname{gh}:W(R)\longrightarrow R^{\mathbb N}$, given by
\[
\operatorname{gh}(r_0,r_1,r_2,\ldots)=\bigl(w_0(r_0),w_1(r_0,r_1),w_2(r_0,r_1,r_2),\ldots\bigr),
\]
\end{definition}

The universal addition and multiplication polynomials above are characterized by the property that the ghost map is a ring homomorphism, where $R^{\mathbb N}$ has componentwise addition and multiplication:
\[
\operatorname{gh}(r+s)=\operatorname{gh}(r)+\operatorname{gh}(s), \qquad \operatorname{gh}(rs)=\operatorname{gh}(r)\operatorname{gh}(s).
\]
In particular $\operatorname{gh}$ defines a natural transformation $\operatorname{gh}\colon W(-) \to (-)^\N$ of functors from commutative rings to commutative rings.

\begin{proposition} \label{prop-witt-presentation-general}
	Let $\tilde R$ be a $p$-torsionfree ring for which $R \ceq \tilde R / p \tilde R$ is reduced, and let $n \geq 1$ be an integer. Then the $(n-1)$-th ghost map induces an injective ring homomorphism
	$$\begin{tikzcd}
		w_{n-1}' : W_n(R) \arrow[r, hookrightarrow] & \tilde R / p^n \tilde R
		\end{tikzcd}$$
		whose image is 
        \begin{equation} \label{eqn-image-of-ghost}
            (\tilde R^{p^{n-1}} + p \tilde R^{p^{n-2}} + \cdots + p^{n-1} \tilde R) / p^n \tilde R,
        \end{equation}
        where we write $\tilde R^{p^i}$ for the \emph{set} of $p^i$-th powers in $\tilde R$.
\end{proposition}
\begin{proof}
    By \autoref{ghost map def}, 
    whenever $x_0, \dots , x_{n-1} \in p \tilde R$ we get $w_{n-1}(x_0, \dots , x_{n-1}) \in p^n \tilde R$. 
	This shows that the composition $W_n(\tilde R) \xrightarrow{w_{n-1}} \tilde R \to \tilde R / p^n \tilde R$ factors through $W_n(R)$ as claimed. The claim about the image is clear from the description of $w_{n-1}$ given above.

	It remains to show that $w_{n-1}'$ is injective; i.e., that given $x_0, \dots, x_{n-1} \in \tilde R$ we have $w_{n-1}(x_0, \dots , x_{n-1}) \in p^n \tilde R$ only if $x_0, \dots , x_{n-1} \in p \tilde R$. We prove this by induction on $n$, with the base case $n = 1$ being clear.

	Observe that $w_{n-1}(x_0, \dots, x_{n-1}) \equiv x_0^{p^{n-1}} \mod p \tilde R$, and from the assumption that $\tilde R / p \tilde R$ is reduced we obtain that $x_0 \in p \tilde R$, which in turn gives 
	$$0 \equiv w_{n-1}(x_0, \dots , x_{n-1}) \equiv p w_{n-2} (x_1, \dots , x_{n-1}) \mod p^n \tilde R.$$
	Using the $p$-torsionfreeness of $\tilde R$, we conclude that $w_{n-2}(x_1, \dots , x_{n-1}) \in p^{n-1} \tilde R$, and now the claim follows from the induction hypothesis.
\end{proof}

\begin{corollary}
	The subset of $\tilde R / p^n \tilde R$ given by \cref{eqn-image-of-ghost} is a subring of $\tilde R / p^n \tilde R$.
\end{corollary}

Next, we apply \cref{prop-witt-presentation-general} to the special case of polynomial rings. This allows us to obtain a description of their rings of Witt vectors that is amenable to computations.

\begin{theorem} \label{thm-witt-presentation-poly}
	Let $R \ceq \F_p [X_1, \dots , X_d]$ be a polynomial ring over $\F_p$ and $n \geq 1$ be an integer. Then $W_n(R)$ is isomorphic to the $(\Z / p^n)$-subalgebra of the polynomial ring $(\Z / p^n)[Y_1, \dots , Y_d]$ generated by the elements
	\begin{equation*}
	G_n \ceq \biggl\{ Y_1^{p^{n-1}}, \dots , Y_d^{p^{n-1}} \biggr\} \cup \biggl\{ p^j (Y_1^{i_1} Y_2^{i_2} \dots Y_d^{i_d})^{p^{n-1-j}} \ \bigg| \ {{0 \leq j \leq n-1} \atop { 0 \leq i_k \leq p^j - 1} }\biggr\}.
	\end{equation*}
\end{theorem}

Note that this list of generators has some redundancies; for example, if $d \geq 2$ and $n=3$ one could always write $p^2 Y_1^p Y_2^p = p ( p Y_1^p Y_2^p)$. 

Another way to think of this set of monomials is as follows. Set $q = p^{n-1}$, consider the polynomial ring
$$\tilde R \ceq \Z [X_1^{1/q}, \dots , X_d^{1/q}]$$
in the formal variables $X_i^{1/q}$. Then, after identifying $Y_i = X_i^{1/q}$, the set $G_n$ gives generators for the subring of $\tilde R / p^n \tilde R$ that consists of  images of polynomials $P(Y_i) = P(X_i^{1/q})$ satisfying 
\[
X_i\frac{\partial P}{\partial X_i} \in \tilde R;
\]
see \cite[\S~3.3]{ChambertLoirSurvey}.

\begin{proof}
Consider the ring $\tilde R \ceq \Z[Y_1, \dots, Y_d]$ so that, by identifying $X_i \in R$ with the class of $Y_i \in \tilde R$, we have $R = \tilde R / p \tilde R$. For every integer $n \geq 1$, let $\tilde R_n \ceq \tilde R / p^n \tilde R$, and let $S_n$ denote the $(\Z / p^n)$-subalgebra of $\tilde{R}_n$ generated by $G_n$. By \cref{prop-witt-presentation-general} we may identify $W_n(R)$ with its image in $\tilde R_n$; that is, with the subring of $\tilde R_n$ given in \cref{eqn-image-of-ghost}. We have ring inclusions $S_n \subseteq W_n(R) \subseteq \tilde R_n$, and our goal is to show that the first one is an equality. We do this by induction on $n$, with the base case $n = 1$ being clear. 

For $n \geq 2$, we claim that both $S_n$ and $W_n(R)$ contain $p^{n-1} \tilde R_n$ as an ideal, the latter being clear from \cref{eqn-image-of-ghost}. As for $S_n$, it suffices to show that $S_n$ contains every element of the form $p^{n-1} Y_1^{a_1} \cdots Y_d^{a_d}$, and by writing $a_i = p^{n-1} b_i + c_i$ for integers $b_i, c_i$ with $0 \leq c_i < p^{n-1}$, we have
\[
p^{n-1} Y_1^{a_1} \cdots Y_d^{a_d} = (Y_1^{p^{n-1}})^{b_1} \cdots (Y_d^{p^{n-1}})^{b_d} \cdot (p^{n-1} Y_1^{c_1} \cdots Y_d^{c_d})
\]
we can write this monomial as a product of elements of $G_n$, thus proving the claim.

By killing the ideal $p^{n-1} \tilde R_n$, we obtain inclusions $S_n / p^{n-1} \tilde R_n \subseteq W_n(R) / p^{n-1} \tilde R_n \subseteq \tilde R_{n-1}$, and it suffices to show that the first one is an equality. To prove this, let $\phi\colon \tilde R_{n-1}  \to \tilde R_{n-1}$ denote the lift of Frobenius given by $\phi(Y_i) = Y_i^p$. Note that in $\tilde R_{n-1}$ we have $p^{n-1} = 0$, so we get that $\phi(G_{n-1})$ and $G_{n}$ generate the same algebra after reduction modulo $p^{n-1}$ and $S_n / p^{n-1} \tilde R_n$ identifies with $\phi(S_{n-1})$. 

Next, note that we have two commutative diagrams
\[
\begin{tikzcd}[column sep = large]
	W_{n-1} (R) \arrow[r, "w_{n-2}'"] \arrow[d, "F"'] & \tilde R_{n-1} \arrow[d, "\phi"] & \null & W_n (R) \arrow[r, "w_{n-1}'"] \arrow[d, "\mathrm{tr}\circ F"'] & \tilde R_n  \arrow[d] \\
    W_{n-1} (R) \arrow[r, "w_{n-2}'"] & \tilde R_{n-1} & \null &  W_{n-1} (R) \arrow[r, "w_{n-2}'"']& \tilde R_{n-1},
\end{tikzcd}
\]
in which $F$ stands for the Frobenius morphism and $\mathrm{tr}$ truncation.
Indeed, the commutativity of the left diagram follows from functoriality of ghost maps (applied to the map $\phi: \tilde R \to \tilde R$) together with the fact that $\phi$ reduces to Frobenius on $R$, and the commutativity of the right diagram is checked by inspection.

From these two diagrams, we deduce that $W_n(R) / p^{n-1} \tilde R_n$ identifies with $\phi(W_{n-1} (R))$. Applying the induction hypothesis, we get 
\[
{S_n}/{p^{n-1}\widetilde R_n}
=
\phi(S_{n-1})
=
\phi\bigl(W_{n-1}(R)\bigr)
=
{W_n(R)}/{p^{n-1}\widetilde R_n}
\]
therefore completing the proof of the theorem.
\end{proof}

For $R \ceq \F_p[X_1, \dots , X_d]$ a polynomial ring over $\F_p$, \cref{thm-witt-presentation-poly} allows us to think of $W_n(R)$ as an explicit subring of a polynomial ring $\tilde R_n \ceq (\Z / p^n)[Y_1, \dots , Y_d]$ over $\Z / p^n$, for which polynomial arithmetic is already implemented in existing computer algebra software.
(Though we do note that many capabilities are not yet implemented over $\Z/p^n$-algebras in Macaulay2, such as kernel and primary decomposition.)
Our package exploits this fact to provide algorithms for arithmetic operations of finite length Witt vectors in polynomial rings over $\F_p$. In order to do this, we first need an algorithm that finds preimages of polynomials under the isomorphism given in \cref{thm-witt-presentation-poly}. To describe this algorithm, recall that $R \ceq \F_p[X_1, \dots , X_d]$ is a polynomial ring over $\F_p$. Given an integer $n \geq 1$, we set $\tilde R_n \ceq (\Z / p^n)[Y_1, \dots , Y_d]$, and we have an inclusion $w_{n-1}' \colon W_n(R) \hookrightarrow \tilde R_n$. For $n = 1$, we identify $\tilde R_1$ with $R$ via the unique isomorphism that exchanges the variable $Y_i$ with the variable $X_i$ for every $1 \leq i \leq d$.

\begin{algorithm}[H] 
\caption{Given an element $P \in \tilde R_n \ceq (\Z / p^n)[Y_1, \dots , Y_d]$, determine whether it belongs to the image of $W_n(R)$ under $w_{n-1}'$. If it does, find its preimage.} \label{alg-1}
\begin{algorithmic}[1]
\State initiate an empty list $W$
\State set $Q \ceq P$
\State set $m \ceq n$
\While{$m \geq 1$}
\State set $q \ceq p^{m-1}$
\State set $Q_1$ in $R = \tilde R_1$ to be the mod-$p$ reduction of $Q$
\If{$Q_1$ is not a $q$-th power in $R$}
\State \textbf{break} ``$P$ is not in the image of $W_n(R)$"
\Else
\State set $f \in R$ to be the $q$-th root of $Q_1$
\State append $f$ to the right of $W$
\State choose a preimage $F$ of $f$ under $\tilde R_m\twoheadrightarrow R $
\State set $Q' \in \tilde R_{m}$ to be an element for which $Q - F^q = p Q'$
\State replace $Q $ by  the image of $Q'$ in $\tilde R_{m-1}$
\State set $m \ceq m-1$
\EndIf
\EndWhile
\State we conclude that $P$ is in the image of $W_{n}(R)$, with preimage $W \in W_n(R)$.
\end{algorithmic}
\end{algorithm}

Note that \cref{alg-1} relies on being able to compute $q$-th roots of elements of $R = \F_p[X_1, \dots , X_d]$, for $q$ a power of $p$. This is easy: given a polynomial $f = \sum_\alpha a_\alpha X_1^{\alpha_1} \cdots X_d^{\alpha_d}$ with all $a_\alpha \neq 0$, we know that $f$ is a $q$-th power precisely when all the $\alpha_i$ are divisible by $q$ and, when this is the case, we have $f^{1/q} = \sum_\alpha a_\alpha X_1^{\alpha_1 / q} \cdots X_d^{\alpha_d / q}$.

With the above algorithm in hand, our strategy for arithmetic operations is as follows. Note that by an arithmetic operation on a ring $R$ we mean a function $\mathsf P \colon R \times \cdots \times R \to R$ that commutes with all ring homomorphisms $R \to S$. For example, given $t_1, t_2 \in W_n(R)$, we may want to compute $\mathsf P(t_1, t_2) = t_1 + t_2$ or $\mathsf P(t_1, t_2) = t_1 \cdot t_2$.

\begin{algorithm}[H]
\caption{For $R \ceq \F_p[X_1, \dots, X_d]$, suppose $(t_i) \subseteq W_n(R) = R \times \cdots \times R$ is a finite collection of Witt vectors of length $n$ over $R$, given as tuples of $n$ elements of $R$. Let $\mathsf P$ be some arithmetic operation that one can apply to the collection $(t_i)$, and which is implemented for polynomials in $\tilde R_n$. To explicitly compute $\mathsf P(t_i) \in W_n(R) = R \times \cdots \times R$ as a tuple, we do the following.} \label{alg-2}
\begin{algorithmic}[1]
\State For every $i$, choose a preimage $\tilde t_i$ under the quotient map $W_n(\tilde R_n) \twoheadrightarrow W_n(R)$
\State For every $i$, compute $w_{n-1} (\tilde t_i) \in \tilde R_n$ by using (\ref{ghost map def}); this gives the image of $t_i$ in $\tilde R_n$ under the isomorphism given by \cref{thm-witt-presentation-poly}. 
\State Compute $\mathsf P (\tilde t_i)$ in $\tilde R_n$
\State Find the preimage of $\mathsf P (\tilde t_i)$ in $W_n(R)$ under the isomorphism from \cref{thm-witt-presentation-poly}
\end{algorithmic}
\end{algorithm}

\begin{remark}
\cref{alg-1} and \cref{alg-2} allow for computing arithmetic operations in $W_n(R)$ where $R$ is an arbitrary finite type algebra over an arbitrary finite field. Indeed, any such algebra is of finite type over $\F_p$, and can be written as a quotient of $\F_p[X_1, \dots , X_d]$. After finding lifts, all arithmetic operations can be performed in this ambient ring. This functionality is implemented in our \texttt{Macaulay2} package.
\end{remark}

\section{Details on the \texttt{Macaulay2} package}
Let $R$ be a finitely generated $\F_p$-algebra; the package expects that \texttt{R} is a quotient of a polynomial ring \texttt{S} and that \texttt{coefficientRing S} is either \texttt{ZZ/p} or \texttt{GF p\^{}e} for some $e$, and will return an error if this is not the case.

\subsection{Witt rings and Witt ring elements}

From the ring $R$, the user can create the corresponding ring of length-$n$ Witt vectors via \texttt{witt(n, R)}; the result is either a \texttt{WittPolynomialRing} or a \texttt{WittQuotientRing}, depending on whether $R$ is a polynomial ring over $\F_p$ or a quotient of such a ring. 

\begin{remark}
    If $R=\F_{p^m}[x_1,\dots,x_n]$ for $m>1$, then the resulting calculations are carried out in the background by treating $R$ as $\F_p[a,x_1,\dots,x_n]/f(a)$, with $f$ an irreducible polynomial of degree $m$; however, the data type of \texttt{witt(n,R)} is still a \texttt{WittPolynomialRing} to avoid mathematical confusion. 
\end{remark}

Elements of \texttt{witt(n, R)} (i.e., elements of type \texttt{WittRingElement}) are created via \texttt{witt\{r1,..,rn\}}, where \texttt{ri} are elements of $R$. 
The underlying tuple can be recovered from a \texttt{WittRingElement} via \texttt{toList}, and individual entries can be accessed via \texttt{\#n}.
The ambient Witt ring can be recovered from a \texttt{WittRingElement} by \texttt{ring}, and the underlying ring $R$ from \texttt{witt(n,R)} via \texttt{unWitt}. 

\begin{example}[Creation and description of Witt ring elements]
\
    \begin{verbatim}
    i1 : R = (GF 25)[x,y]/(x^2-y^3);
    i2 : w = witt{2*x,y}
    o2 = {2x, y}
    o2 : WittRingElement
    i3 : toList w
    o3 = {2x, y}
    o3 : List
    i4 : w#1
    o4 = y
    o4 : R
    i5 : W = ring w
    o5 = Witt (R)
             2
    o5 : WittQuotientRing
    i6 : unWitt W
    o6 = R
    o6 : QuotientRing
    \end{verbatim}
\end{example}

\subsection{Arithmetic of Witt vectors}
Basic ring operations like \texttt{+}, \texttt{-}, \texttt{*} are defined on pairs of elements of type \texttt{WittRingElement} that live in the same Witt ring (which is to say, which are defined as \texttt{witt} of lists of elements of $R$ of the same length). 

\begin{example}[Addition and multiplication of Witt vectors]
\
    \begin{verbatim}
    i1 : S = (ZZ/5)[x_1,x_2,y_1,y_2];
    i2 : W2S = witt(2,S)
    o2 = Witt (S)
             2
    o2 : WittPolynomialRing
    i3 : w1 = witt{x_1,x_2}
    o3 = {x , x }
           1   2
    o3 : WittRingElement
    i4 : w2 = witt{y_1,y_2};
    i5 : w1+w2
                      4       3 2     2 3      4
    o5 = {x  + y , - x y  - 2x y  - 2x y  - x y  + x  + y }
           1    1     1 1     1 1     1 1    1 1    2    2
    o5 : WittRingElement
    i6 : w1*w2
                 5    5
    o6 = {x y , x y  + x y }
           1 1   2 1    1 2
    o6 : WittRingElement
    i7 : -2*w1
                     5
    o7 = {-2x , x  - 2x }
             1   1     2
    o7 : WittRingElement
    \end{verbatim}
\end{example}

\subsection{Witt ring maps and functoriality}
A map $f: R\to S$ induces maps $W_n(f) : W_n(R)\to W_n(S)$ for all $n$ (in the tuple representative, these maps are just applying $f$ componentwise).  This map is obtained from a \texttt{RingMap}, say \texttt{f}, by \texttt{witt(n, f)}. The methods \texttt{source, target} will return the source and target, and the underlying ring map $f$ can be obtained as \texttt{baseMap}.

\begin{example}[Witt ring maps]
\
\begin{verbatim}
    i1 : R = (ZZ/5)[x,y];
    i2 : S = (ZZ/5)[a,b,c,d];
    i3 : f = map(S, R, {a*b, c*d});
    o3 : RingMap S <-- R
    i4 : Wf = witt(2, f)
    o4 = WittRingMap Witt (S) <-- Witt (R)
                          2            2
    o4 : WittRingMap
    i5 : (source Wf, target Wf)
    o5 = (Witt (R), Witt (S))
              2         2
    o5 : Sequence
    i6 : baseMap Wf === f
    o6 = true
\end{verbatim}
\end{example}

\subsection{Witt ring operations}
The usual Witt ring operations of Frobenius and Verschiebung are implemented:
The Verschiebung (or shift) operator acts on elements of a Witt ring; \texttt{verschiebung w} acts by shifting \texttt{w} once and 
\texttt{verschiebung(n, w)} shifts it $n$ times. 
The $e$-th Witt Frobenius of a Witt ring \texttt{W = witt(n,R) } can be constructed as a \texttt{WittRingMap} by either \texttt{wittFrobenius(e, W)} or directly from $R$ by \texttt{wittFrobenius(e, n, R)}; if $e$ is not given as an argument, it is taken to be 1.
 The method \texttt{wittFrobenius} can act on objects of class \texttt{WittRingElement}. 
Finally, the truncation maps $W_m(R) \to W_{n}(R)$ for $n\leq m$ are implemented both on elements and as \texttt{WittRingMap}s via \texttt{truncate(n, WittRingElement)}, \texttt{truncate(n, WittPolynomialRing)} and \texttt{truncate(n, WittQuotientRing)}. 
 
\begin{example}[Frobenius, Verschiebung, and truncations in Witt vectors]
\ 
    \begin{verbatim}
    i1 : S = (ZZ/5)[x_1,x_2,x_3,y_1,y_2,y_3]
    o1 = S
    o1 : PolynomialRing
    i2 : W3S = witt(3,S)
    o2 = Witt (S)
             3
    o2 : WittPolynomialRing
    i3 : w = witt{x_1,x_2,x_3}
    o3 = {x , x , x }
           1   2   3
    o3 : WittRingElement
    i4 : wittFrobenius(w)
            5   5   5
    o4 = {x , x , x }
           1   2   3
    o4 : WittRingElement
    i5 : F = wittFrobenius(1, W3S)
    o5 = WittRingMap Witt (S) <-- Witt (S)
                         3            3
    o5 : WittRingMap
    i6 : wittFrobenius(w) == F(w)
    o6 = true
    i7 : Vw = verschiebung(w)
    o7 = {0, x , x , x }
              1   2   3
    o7 : WittRingElement
    i8 : truncate(3, Vw)
    o8 = {0, x , x }
              1   2
    o8 : WittRingElement
    i9 : truncate(3, ring Vw)
    o9 = WittRingMap Witt (S) <-- Witt (S)
                         3            4
    o9 : WittRingMap
    \end{verbatim}
\end{example}

\subsection{Explicit calculation of $W_n(R)$}
We also implement a calculation of $W_n(R)$ as an explicit $\Z/p^n$-algebra, conversion of tuples in $W_n(R)$ to elements of this explicit $\Z/p^n$-algebra, and representations of Witt ring maps as maps on these algebras.
To obtain a description of \texttt{W = witt(n, R)} as a $\Z/p^n$-algebra, one writes \texttt{explicit W}. (Caution: these rings have many generators and relations even for small $n$ and $R$ with few generators!)
Likewise, if \texttt{Wf} is a \texttt{WittRingMap} then \texttt{explicit Wf} returns the map from \texttt{explicit source Wf} to \texttt{explicit target Wf}. 
Finally, to obtain the element of \texttt{explicit witt(n,R)} corresponding to a \texttt{WittRingElement} $w$, use \texttt{wittTupleToRing w}. 

\begin{example}[explicit representatives of Witt rings]
\ 
    \begin{verbatim}
    i1 : R = (ZZ/2)[x]
    i2 : E = explicit witt(2, R)
    
          ZZ[T        , T        ]
              {0, {1}}   {1, {1}}
    o2 = --------------------------
                          2
         (4, 2T        , T        )
               {1, {1}}   {1, {1}}
    i3 : explicit wittFrobenius witt(2, R)
                      2
    o3 = map (E, E, {T        , 2T        })
                      {0, {1}}    {0, {1}}
    o3 : RingMap E <-- E
    i4 : wittTupleToRing witt{x, 0}
    o4 = T
          {0, {1}}
    o4 : E
    \end{verbatim}
\end{example}

\subsection{Ideals in Witt rings}\label{section witt ideals}
Ideals in Witt rings are also implemented; the user can specify such an ideal by giving a list of generators to the method \texttt{wittIdeal}, and the generators of a Witt ideal can be recovered via \texttt{generators} exactly as for usual ideals.
Standard operations such as ideal products,  addition, and comparison are implemented.
Moreover, a minimal generating set for a Witt ring ideal can be found via \texttt{trim}.
Given a Witt ring ideal, the user can find explicit generators for the ideal in the explicit presentation of the Witt ring via \texttt{explicit}.

\begin{example}[Witt ring ideals]\
    \begin{verbatim}
    i1 : S = (ZZ/2)[x,y];
    i2 : WI = wittIdeal(witt{x, x}, witt{y, y}, witt{x + y, x*y + x + y})
    o2 = ideal ({x, x}, {y, y}, {x + y, x*y + x + y})
    o2 : WittIdeal
    i3 : WI' = trim WI
    o3 = ideal ({y, y}, {x, x})
    o3 : WittIdeal
    i4 : WI == WI'
    o4 = true
    i5 : WI^2
                  2             2       2     2
    o5 = ideal ({y , 0}, {x*y, x y + x*y }, {x , 0})
    o5 : WittIdeal
    i6 : explicit WI'

    o6 : ideal (T            + T           , T            + T           )
                 {0, {0, 1}}    {1, {0, 1}}   {0, {1, 0}}    {1, {1, 0}}
    \end{verbatim}
\end{example}

\section{Finding Frobenius lifts}\label{section frobeniuslift}
In this section, we explain how we can compute the set of possible lifts of Frobenius on an $\F_p$-algebra $A$ to a  specific lift $B$ to a (flat) $W_2(\F_p)$-algebra. 

\begin{remark}
In a similar direction,
    Zdanowicz discovered a criterion to check whether a \emph{complete intersection ring} admits a flat lifting to $W_2(k)$ along with a compatible lifting of the Frobenius (and wrote Macaulay2 code to find the explicit flat lifting and lift of Frobenius) \cite[Proposition~7.2.4]{ZdanowiczThesis}.
\end{remark}

Let $A$ be a ring of characteristic $p>0$ that is a quotient of a polynomial ring over $\F_p$, say \[
A=\F_p[x_1,\ldots,x_n]/(g_1,\ldots,g_r).\] 
The ``default'' lift used is $B\coloneqq\Z/p^2[x_1,\ldots,x_n]/(\tilde{g}_1,\ldots,\tilde{g}_r)$ where $\tilde{g}_i$ is the polynomial in $\Z/p^2[x_1,\ldots,x_n]$ with coefficients in $0,\ldots,p-1$ such that $\tilde{g}_i\mod p=g_i$.
We set $I:=(g_1, \ldots, g_r)$ and $\tilde{I} := (\tilde{g}_1, \ldots, \tilde{g}_r)$.
Our methods allow the user to specify a different choice of lift via \texttt{PerturbationTerm} (see  \cref{example perturbation term}); at the moment, there is not functionality for lifts that are not given by lifting generators of the defining ideal.



A lift of the Frobenius to a ring map $\phi:B\to B$ must by definition for each $b \in B$ satisfy 
$$
\phi(b) = b^p + p b'
$$
for some $b'\in B$; since $B$ is generated by $x_1,\dots,x_n$, the lift of Frobenius is determined uniquely by $\phi(x_1),\dots,\phi(x_n)$, and thus equivalently by the values $x_1',\dots,x_n' \in B$. 
Note also that the ring map $\phi$ does not change if one replaces all $b'$ by $b'+pb''$, since 
$$
p(b'+pb'') = pb'+p^2b'' = pb'. 
$$
We may thus view the elements $x_i'$ as living in $B/pB = A$ rather than $B$. 

If $B$ is a polynomial ring, say $B=\Z/p^2[x_1,\ldots, x_n]$ then any choice of $x_i' \in A$ gives a well-defined lift of Frobenius $\phi:B\to B$.
If $B$ is instead a quotient of a polynomial ring, say $B=\Z/p^2[x_1,\ldots,x_n]/(\tilde g_1,\ldots, \tilde g_r)$, which is flat over $\Z/p^2\Z$, then in order to be well-defined, the map $\phi$ must send
$(\tilde g_1,\dots,\tilde g_r) $ to itself; we now find conditions such that this holds.

Since $\tilde g_i = g_i +p h_i$, we can write
$$
\phi (\tilde g_i) = \phi(g_i) + p\phi(h_i) = \phi(g_i) +ph_i^p.
$$
Taylor expanding and writing $\mathbf{x}=(x_1,\ldots,x_n)$ and $\mathbf{x}'=(x_1',\ldots,x_n')$ gives
$$
\phi( g_i ) =  g_i(\phi(\mathbf{x})) =
 g_i(\mathbf{x}^p +p \mathbf{x'}) =  g_i(\mathbf{x}^p) + p \sum_j \frac{\partial  g_i}{\partial x_j}(\mathbf{x}^p)\cdot x_j'.
$$
Now, by setting $\delta(g_j) := (g_j(\mathbf x^p) - g_j(\mathbf x)^p)/p$, we can write this as 
$$
g_i^p + p\biggl( \delta(g_i) + \sum_j \frac{\partial g_i}{\partial x_j}(\mathbf x^p)\cdot x_j' \biggr).
$$
To summarize, we have that 
$$
\phi (\tilde g_i) = g_i^p + p\biggl( \delta(g_i) + h_i^p+  \sum_j \frac{\partial g_i}{\partial x_j}(\mathbf x^p)\cdot x_j'\biggr).
$$
Since $g_i^p  = \tilde g_i^p \mod p^2$, we have that $g_i^p \in \tilde I$ automatically, and thus $\phi$ preserves $\tilde I$ if and only if 
$$
p \biggl( \delta(g_i) + h_i^p+  \sum_j \frac{\partial g_i}{\partial x_j}(\mathbf{x}^p)\cdot x_j'\biggr) \in \tilde I 
$$
for all $i$.
By flatness over $\Z/p^2\Z$, this holds if and only if 
\begin{equation}\label{eqn delta is 0}
\delta(g_i) + h_i^p+  \sum_j \frac{\partial g_i}{\partial x_j}(\mathbf{x}^p)\cdot x_j' \in I;
\end{equation}
this is now an equation entirely in $A$; taking the set of these equations for all $i$  yields the constraints on the $x_j'$ that need to be satisfied for $\phi$ to give a valid lift of the Frobenius on $B$.
(Note also that 
$\delta(g_i) + h_i^p = \delta(g_i+ph_i)$.)



The \texttt{findFrobeniusLiftConstraints} method finds the constraints imposed by \autoref{eqn delta is 0} on the $x_i'$. To  avoid confusion between variables $x_i$ in the ring whose Frobenius lift we want to find and the corresponding variables $x_i'$, we have chosen to denote $x_i'$ by $aa_i$ (where $i$ starts at 0).
The output is thus an equation in the $aa_i$ and original variables $x_i$.

One can produce random Frobenius lifts by using the \texttt{findFrobeniusLift} method; since the $x_i' $ can be of arbitrarily high degree, the user must specify a maximal degree $d$ of the terms appearing in the $x_i'$. 

\begin{example} We consider the case of the cusp for $p=2$.
    \begin{verbatim}
    i1 : S = (ZZ/2)[x,y];
    i2 : f = x^2-y^3;
    i3 : findFrobeniusLiftConstraints(S/f)
                    2     4
    o3 = ideal (aa x y + x )
                  1
                  
                  ZZ
                  --[aa , aa , x, y]
                   2   0    1
    o3 : Ideal of ------------------
                          3    2
                         y  + x
    \end{verbatim}
The equation in $\texttt{o3}$ tells us that the lift $aa_0=x'$ can be anything, but $aa_1=y'$ must satisfy $y'x^2y+x^4 \in (x^2+y^3)$. Replacing $x^2$ by $y^3$ gives $y'y^4+y^6 \in (x^2+y^3)$ so $y'\in y^2 +(x^2+y^3)$. 

The user can then attempt to find a random lift by using \texttt{findFrobeniusLift(d,f)}, where $d$ specifies the highest degree term allowed:
\begin{verbatim}
    i4 : findFrobeniusLift(2, S/f)
                    2	
    o4 = {x*y + y, y }
    o4 : List
    i5 : findFrobeniusLift(3, S/f)
           3    2       2    3    2                 3    2    2
    o5 = {x  + x y + x*y  + y  + x  + x*y + x + y, y  + x  + y }
 \end{verbatim}
The first line tells us that $x' = xy+y$, $y'=y^2$ define a valid lift of Frobenius on $\Z/4\Z[x,y]/(x^2+y^3)$.
Because we saw above that the only constraint is $y' \in y^2+(x^2+y^3)$, raising the degree from 2 to 3 
can change $y'$ only by an element of $(f)$
as one can see when trying \texttt{findFrobeniusLift(3,f)}.
\end{example}

\begin{example}
    Some rings (e.g., monomial rings) will have lifts of Frobenius given by simply setting $x_i'=0$ (i.e., using the ``naive'' lift $x_i\mapsto x_i^p$). If the user wants to look for other lifts, they can set the \texttt{Nontrivial} flag to true:
\begin{verbatim}
    i1 : S = (ZZ/2)[x,y]
    o1 = S
    o1 : PolynomialRing
    i2 : I = ideal(x*y)
    o2 = ideal(x*y)
    o2 : Ideal of S
    i3 : findFrobeniusLift(2, I)
    o3 = {0, 0}
    i4 : findFrobeniusLift(2, I, Nontrivial => true)
    o4 = {x, y}
\end{verbatim}
Thus, one finds that both $x\mapsto x^2,y\mapsto y^2$ and  $x\mapsto x^2 + px$ and $y\mapsto y^2+py$ are lift of Frobenius $S/I$.
\end{example}

The method \texttt{findFrobeniusLiftConstraints} gives the equations which must be satisfied by the choice of $aa_i$. If one wants to parametrize the $aa_i$ by their coefficients (up to a given degree) and translate the equations of the $aa_i$ into an equation of their coefficients,
one can use the method \texttt{parametrizedLifts}. 
If $R = S/(I)$ for a polynomial ring $S=k[x_1,\ldots,x_n]$ and one represents the choice of a Frobenius lift on $W_2(k)[x_1,\ldots,x_n]$ by the choice of $x_i'$ of degree $\leq d$, this method returns the equations that the coefficients of those polynomials have to satisfy for the resulting Frobenius lift to descend to $W_2(k)[x_1,\ldots,x_n]/(I)$. 
The output is an ideal $J$ in variables $c_{{a_1..a_n},j}$, where $c_{{a_1..a_n},j}$ is the coefficient of $x_1^{a_1}..x_n^{a_n}$ 
in the correction polynomial $x_j'$
; the exponents $a_i$ satisfy $\sum a_i \leq d$.

\begin{example}\label{example create equations no perturbation}\ {\small
\begin{verbatim}
    i1 : S = (ZZ/2)[x,y];
    i2 : I = ideal(x+y);
    i3 : J = parametrizedLifts(2, I);
    i4 : netList  J_*
         +-------------------------------------------------------------------------------------+
    o4 = |c           + c           + c           + c           + c           + c           + 1|
         | {0, 2},{1}    {0, 2},{2}    {1, 1},{1}    {1, 1},{2}    {2, 0},{1}    {2, 0},{2}    |
         +-------------------------------------------------------------------------------------+
         |c           + c           + c           + c                                          |
         | {0, 1},{1}    {0, 1},{2}    {1, 0},{1}    {1, 0},{2}                                |
         +-------------------------------------------------------------------------------------+
         |c           + c                                                                      |
         | {0, 0},{1}    {0, 0},{2}                                                            |
         +-------------------------------------------------------------------------------------+
\end{verbatim}}
This says that the coefficients of the constant and linear terms of $x'$ and $y'$ must sum to $0$ in each degree, and the coefficients of the degree-2 terms must sum to $1$.
\end{example}

To change the choice of Frobenius lift, one can change the \texttt{PerturbationTerm}. Say \[
A=\F_p[x_1,\ldots,x_n]/(g_1,\ldots,g_r).\] The default lift (e.g., the one used above) is $B\coloneqq\Z/p^2[x_1,\ldots,x_n]/(\tilde{g}_1,\ldots,\tilde{g}_r)$ where $\tilde{g}_i$ is the polynomial in $\Z/p^2[x_1,\ldots,x_n]$ with coefficients in $0,\ldots,p-1$ such that $\tilde{g}_i\mod p=g_i$. The \texttt{PerturbationTerm} option allows for other choices of lift:
setting \texttt{PerturbationTerm => $\{\alpha_1,\alpha_2,\ldots,\alpha_r\}$} specifies the lifted ring 
$\Z/p^2[x_1,\ldots,x_n]/(\tilde{g}_1+p\alpha_1, \tilde{g}_2+p\alpha_2,\ldots, \tilde{g}_r+p\alpha_r)$.
By default \texttt{PerturbationTerm} is null (i.e., corresponding to $\alpha_i=0$); note that the value of \texttt{PerturbationTerm} can affect the existence of a Frobenius lift.

One can also choose the polynomials that give the image of $x_j$ under the Frobenius lift to be homogeneous by setting the \texttt{Homogeneous} option to true. It is false by default. 
This is useful when looking for lifts of Frobenius on coordinate rings of projective varieties, and  can also be quite useful since it reduces the search space and therefore allows \texttt{findFrobeniusLift} to find an answer faster.

\begin{example}\label{example perturbation term}
    Here is a continuation of \autoref{example create equations no perturbation} with a different perturbation term and with the \texttt{Homogeneous} option set to true.
   { \small
\begin{verbatim}
    i4 : J = parametrizedLifts(2, I, Homogeneous=>true, PerturbationTerm=>{y})  
    o4 = ideal(c           + c           + c           + c           + c           + c          )
                {0, 2},{1}    {0, 2},{2}    {1, 1},{1}    {1, 1},{2}    {2, 0},{1}    {2, 0},{2}

                  ZZ
    o4 : Ideal of --[c          , c          , c          , c          , c          , c          ]
                   2  {0, 2},{1}   {0, 2},{2}   {1, 1},{1}   {1, 1},{2}   {2, 0},{1}   {2, 0},{2}
\end{verbatim}}
This means that the coefficients of the degree-2 terms in \(x'\) and \(y'\) must sum to \(0\).
\end{example}

The choice of \texttt{PerturbationTerm} can affect the existence of a Frobenius lift:

\begin{example}
One source of such examples is in the theory of elliptic curves: if $X$ is an ordinary elliptic curve over $\mathbb F_p$, the theory of Serre--Tate (see, e.g., \cite{LubinSerreTate1964}) lifts says there is a unique  lift of $X$ (up to isomorphism) to a $W_2(k)$-algebra admitting a lift of Frobenius. In particular, different choices of perturbation term will admit or not admit lifts of Frobenius:
\begin{verbatim}
    i1 : S = (ZZ/2)[x,y,z];
    i2 : I =  ideal(x^3+x^2*z+x*y*z+x*z^2+y^2*z+y*z^2);
    i3 : L = findFrobeniusLift(2, I,  Homogeneous => true, PerturbationTerm => {z^3})
           2          2          2
    o3 = {x  + x*y + y  + y*z + z , x*z, x*z}
    o3 : List
\end{verbatim}
On the other hand, letting the \texttt{PerturbationTerm} be the default will not give a lift. This can be seen by Serre--Tate or by running \texttt{parametrizedLifts}.
\begin{verbatim}
    i5 : dim parametrizedLifts(2,I)
    o5 : -1
\end{verbatim}
Per Macaulay2 conventions, this means the solution set is empty. In particular, since the \texttt{findFrobeniusLift} method only tries random polynomials until it finds a solution, \texttt{findFrobeniusLift(2, I)} would run indefinitely. 
\end{example}

\section{Quasi-$F$-splittings}\label{section quasiFSplitting}
We now explain how to use our code to detect whether or not a given variety is quasi-$F$-split \emph{at the origin} and if so, what is its quasi-$F$-splitting height.
We recall the definition of a quasi-$F$-splitting. 
\begin{definition}[\cite{YobukoQuasiFSplittingLiftingCY}]
Let $R$ be a commutative Noetherian ring of characteristic $p>0$ and let $W_n(R)$ be its $n$-th Witt ring for any $n\in \N$. Let $\pi\colon W_n(R)\to R$ be the projection map onto the first coordinate and let $F\colon W_n(R)\to W_n(R)$ be the Frobenius map. The \emph{quasi-$F$-split} height $\height(R)$ of $R$ is the infimum number $n>0$ such that there exists a $W_n(R)$-module homomorphism $\phi$ which makes the following diagram commutative
\begin{center}
\begin{tikzcd}
W_n(R) \arrow[d,"\pi"'] \arrow[r, "F"] & F_* W_n(R) \arrow[dl,dashrightarrow, "\phi"]\\
R
\end{tikzcd}
\end{center}
\end{definition}

In \cite{KawakamiTakamatsuYoshikawaFedderQuasi}, Kawakami, Takamatsu, and Yoshikawa showed the existence of a Fedder-type criterion for quasi-$F$-splittings. We implement a version of it for complete intersections in our package. A faster implementation can be done in Julia, see \cite{BatubaraGarzellaPanFastAlgorithmQuasiFSplitting}. 

\begin{remark}
    We note that the only place our code uses computations in rings of Witt vectors is in the computation of $\Delta_1$ that appears in the theorem \cite[Theorem 4.11]{KawakamiTakamatsuYoshikawaFedderQuasi}. This computation could be done without using Witt vectors, as explained in \cite[Section 1.1]{KawakamiTakamatsuYoshikawaFedderQuasi}. We have included it in our package due to the relevance of Witt vectors in quasi-$F$-splitting.
\end{remark}

Let $R=\F_p[x_1,\ldots,x_n]$ and $I\coloneqq(f_1,\ldots,f_m)$ be an ideal of $R$ generated by a regular sequence. 
The method \texttt{fSplittingHeight} 
takes the ideal $I$ and will return the quasi-$F$-splitting height $\height(R/I)$ of $R/I$ \emph{at the origin} if it is less than the \texttt{MaxHeight} options or equal to infinity. It will not return anything otherwise. The default \texttt{MaxHeight} is 102, since Artin-Mazur heights of Calabi-Yau quintic threefolds are bounded above by
$102$, when finite. As shown in \cite[Example 6.7]{KawakamiTakamatsuYoshikawaFedderQuasi} varieties with arbitrarily high quasi-$F$-splitting height exist.

\begin{example}Quasi-$F$-splitting is only checked at the origin. In particular, we might get an \texttt{fSplittingHeight} of $1$ even if the ring is not $F$-pure (as long as it is $F$-pure at the origin).
\begin{verbatim}
    i1 : S = (ZZ/2)[x,y]
    o1 = S
    o1 : PolynomialRing
    i2 : I = ideal(x^2-y^3-1)
                 3    2
    o2 = ideal(y  + x  + 1)
    o2 : Ideal of S
    i3 : fSplittingHeight(I)
    o3 = 1
\end{verbatim}
\end{example}

\begin{remark}
The Fedder-type criterion of \cite[Theorem 4.11]{KawakamiTakamatsuYoshikawaFedderQuasi} works for rings defined over any field $k$ of characteristic $p>0$. However, Macaulay2 only works with the $k$-linear Frobenius whereas the Fedder-type criterion uses an absolute Frobenius. Since we want the two to coincide, we only work over $k=\F_p$.
\end{remark}

\begin{example} [\cite{KawakamiTakamatsuYoshikawaFedderQuasi}, Example 6.2] A Calabi-Yau surface of quasi-$F$-splitting height $5$.
\begin{verbatim}
    i1 : R = (ZZ/3)[x,y,z,w];
    i2 : I = ideal(x^4 + y^4 + z^4 + w^4 + x^3*z + z^3*w + y*z^2*w + y*z*w^2);
               4    4    3     4      2     3         2    4
    o2 = ideal(x  + y  + x z + z  + y*z w + z w + y*z*w  + w )
    o2 : Ideal of R
    i3 : fSplittingHeight(I)
    o3 = 5
\end{verbatim}

\end{example}

Let $R\coloneqq \F_p[x_1,\ldots,x_n]$, let $\fm:= (x_{1}, \ldots, x_n)$, and $f_1,\ldots,f_r$ be a regular sequence in $R$.
Set $I:=(f_1,\ldots,f_r)$ and $f:=f_1\cdots f_r$.
Our computation of $\height (R/I)_{\fm}$ is based on the ideal iteration in \cite[Theorem 4.11]{KawakamiTakamatsuYoshikawaFedderQuasi}.

For $g= \sum_i b_i M_i \in R$,where $b_i \in \F_p$ and the $(M_i)$ are distinct monomials, define $\Delta_1 (g) \in R$ by 
\[
(0,\Delta_1(g))=(g,0)-\Sigma(b_i M_i,0)
\]
in $W_2(R)$.
Before performing the iteration, the code first checks whether $f^{p-1}\notin \fm^{[p]}$, in which case the height is (1).
Otherwise, \cite[Corollary 4.19]{KawakamiTakamatsuYoshikawaFedderQuasi}, applied to $R_{\fm}$, shows that the height is infinite if either $f^{p-2} \in \fm^{[p]}$  or $((f^{p-2})+I^{[p]})
f^{p(p-2)}\Delta_1(f)
\subseteq\mathfrak m^{[p^2]}.$
These are the sufficient conditions checked by our code before starting the iteration. The code also detects infinite height if the sequence of ideals stabilizes while remaining contained in $\fm^{[
p
]}$. However, infinite height may remain undetected within the prescribed MaxHeight bound.

\begin{example} \ 
\begin{verbatim}
    i1 : R = (ZZ/3)[x,y,z,w];
    i2 : I = ideal(x^4 + y^4 + z^4 + w^4);
                4    4    4    4
    o2 = ideal(x  + y  + z  + w  )
    o2 : Ideal of R
    i3 : fSplittingHeight(I)
    o3 = infinity
    o3 : InfiniteNumber
\end{verbatim}
\end{example}

The ideal $I$ can be a complete intersection in place of a hypersurface:

\begin{example}\
\begin{verbatim}
    i1 : S = (ZZ/3)[x,y,z,w,u,s];
    i2 : I = ideal(x*y + y*w^5 + z*w^3 + x*y*z^3 + y*z*w^2,u*z-z*x^3)
                   5        3        2      3           3
    o2 = ideal (y*w  + x*y*z  + y*z*w  + z*w  + x*y, - x z + z*u)
                  ZZ
    o2 : Ideal of --[x..z, w, u, s]
                   3
    i3 : fSplittingHeight(I)
    o3 = 1
\end{verbatim}
\end{example}

\printbibliography

@PREAMBLE{


"\def\cfudot#1{\ifmmode\setbox7\hbox{$\accent"5E#1$}\else
\setbox7\hbox{\accent"5E#1}\penalty 10000\relax\fi\raise 1\ht7
\hbox{\raise.1ex\hbox to 1\wd7{\hss.\hss}}\penalty 10000
\hskip-1\wd7\penalty 10000\box7} "
}

@phdthesis{ZdanowiczThesis,
    author = {Maciej Zdanowicz},
    title = {On liftability of schemes and their Frobenius morphism},
    year         = 2017,
  school       = {University of Warsaw},
  type         = {PhD thesis}
}

@article{WittZyklischeCharp,
 author = {Witt, Ernst},
 title = {Zyklische {K{\"o}rper} und {Algebren} der {Charakteristik} {{\(p\)}} vom {Grad} {{\(p^n\)}}. {Struktur} diskret bewerteter perfekter {K{\"o}rper} mit vollkommenem {Restklassenk{\"o}rper} der {Charakteristik} {{\(p\)}}},
 fjournal = {Journal f{\"u}r die Reine und Angewandte Mathematik},
 journal = {J. Reine Angew. Math.},
 issn = {0075-4102},
 volume = {176},
 pages = {126--140},
 year = {1936},
 url = {https://eudml.org/doc/149985},
}

@ARTICLE{BatubaraGarzellaPanFastAlgorithmQuasiFSplitting,
       author = {{Batubara}, Ryan and {Garzella}, Jack J and {Pan}, Alex},
        title = "{K3 surfaces of any Artin-Mazur height over $\mathbb{F}_5$ and $\mathbb{F}_7$ via Quasi-F-split singularities and GPU acceleration}",
      journal = {arXiv e-prints},
         year = 2025,
        month = feb,
          eid = {arXiv:2502.12428},
        pages = {arXiv:2502.12428},
          doi = {10.48550/arXiv.2502.12428},
archivePrefix = {arXiv},
       eprint = {2502.12428},
 primaryClass = {math.AG},
       adsurl = {https://ui.adsabs.harvard.edu/abs/2025arXiv250212428B}
}

@article {IllusieComplexeDeRhamWittCristalline,
    AUTHOR = {Illusie, Luc},
     TITLE = {Complexe de de~{R}ham-{W}itt et cohomologie cristalline},
   JOURNAL = {Ann. Sci. \'Ecole Norm. Sup. (4)},
  FJOURNAL = {Annales Scientifiques de l'\'Ecole Normale Sup\'erieure.
              Quatri\`eme S\'erie},
    VOLUME = {12},
      YEAR = {1979},
    NUMBER = {4},
     PAGES = {501--661},
      ISSN = {0012-9593},
   MRCLASS = {14F30},
  MRNUMBER = {565469},
MRREVIEWER = {William\ E.\ Lang},
       URL = {http://www.numdam.org/item?id=ASENS_1979_4_12_4_501_0},
}

@ARTICLE{MunozBertrandFasterComputationWittVectors,
       author = {{Mu{\~n}oz-\relax-Bertrand}, Rub{\'e}n},
        title = "{Faster computation of Witt vectors over polynomial rings}",
      journal = {arXiv e-prints},
         year = 2025,
        month = apr,
          eid = {arXiv:2504.01834},
        pages = {arXiv:2504.01834},
          doi = {10.48550/arXiv.2504.01834},
archivePrefix = {arXiv},
       eprint = {2504.01834},
 primaryClass = {math.AC},
       adsurl = {https://ui.adsabs.harvard.edu/abs/2025arXiv250401834M}
}

@article {YobukoQuasiFSplittingLiftingCY,
    AUTHOR = {Yobuko, Fuetaro},
     TITLE = {Quasi-{F}robenius splitting and lifting of {C}alabi-{Y}au
              varieties in characteristic {$p$}},
   JOURNAL = {Math. Z.},
  FJOURNAL = {Mathematische Zeitschrift},
    VOLUME = {292},
      YEAR = {2019},
    NUMBER = {1-2},
     PAGES = {307--316},
      ISSN = {0025-5874,1432-1823},
   MRCLASS = {14J32 (13F35 14G17)},
  MRNUMBER = {3968903},
MRREVIEWER = {Tyler\ L.\ Kelly},
       DOI = {10.1007/s00209-018-2198-7},
       URL = {https://doi.org/10.1007/s00209-018-2198-7},
}

@article {FinottiComputationsWittVectorsLength3,
    AUTHOR = {Finotti, Lu\'is R. A.},
     TITLE = {Computations with {W}itt vectors of length 3},
   JOURNAL = {J. Th\'eor. Nombres Bordeaux},
  FJOURNAL = {Journal de Th\'eorie des Nombres de Bordeaux},
    VOLUME = {23},
      YEAR = {2011},
    NUMBER = {2},
     PAGES = {417--454},
      ISSN = {1246-7405,2118-8572},
   MRCLASS = {13F35},
  MRNUMBER = {2817938},
MRREVIEWER = {Markus\ Szymik},
       DOI = {10.5802/jtnb.770},
       URL = {https://doi.org/10.5802/jtnb.770},
}

@Article{JoyalDeltaAnneauWitt,
 Author = {Joyal, Andr{\'e}},
 Title = {{{\(\delta\)}}-anneaux et vecteurs de Witt. ({{\(\delta\)}}-rings and {Witt} vectors)},
 FJournal = {Comptes Rendus Math{\'e}matiques de l'Acad{\'e}mie des Sciences},
 Journal = {C. R. Math. Acad. Sci., Soc. R. Can.},
 ISSN = {0706-1994},
 Volume = {7},
 Pages = {177--182},
 Year = {1985},
}

@book{JacobsonBasicAlgebraII,
  title={Basic Algebra II: Second Edition},
  author={Jacobson, N.},
  isbn={9780486135212},
  series={Dover Books on Mathematics},
  year={2012},
  publisher={Dover Publications}
}

@BOOK{SerreLocalFields,
  title = {Local fields},
  publisher = {Springer-Verlag},
  year = {1979},
  author = {Serre, Jean-Pierre},
  volume = {67},
  pages = {viii+241},
  series = {Graduate Texts in Mathematics},
  address = {New York},
  note = {Translated from the French by Marvin Jay Greenberg},
  isbn = {0-387-90424-7},
  mrclass = {12Bxx},
  mrnumber = {MR554237 (82e:12016)}
}

@ARTICLE{KawakamiTakamatsuYoshikawaFedderQuasi,
       author = {{Kawakami}, Tatsuro and {Takamatsu}, Teppei and {Yoshikawa}, Shou},
        title = "{Fedder type criteria for quasi-$F$-splitting}",
      journal = {arXiv e-prints},
         year = 2022,
        month = apr,
          eid = {arXiv:2204.10076},
        pages = {arXiv:2204.10076},
          doi = {10.48550/arXiv.2204.10076},
archivePrefix = {arXiv},
       eprint = {2204.10076},
 primaryClass = {math.AG},
       adsurl = {https://ui.adsabs.harvard.edu/abs/2022arXiv220410076K}
}

@article{ChambertLoirSurvey,
  author  = {Chambert-Loir, Antoine},
  title   = {Cohomologie cristalline : un survol},
  journal = {Expositiones Mathematicae},
  volume  = {16},
  number  = {4},
  pages   = {333--382},
  year    = {1998},
}

@misc{LubinSerreTate1964,
  author    = {Lubin, Jonathan and Serre, Jean-Pierre and Tate, John},
  title     = {Elliptic curves and formal groups},
  howpublished = {Mimeographed notes, AMS Summer Institute in Algebraic Geometry},
  year      = {1964},
  address   = {Woods Hole}
}

\end{document}